\documentclass[11pt]{amsart}
\usepackage{amsmath, amsfonts, amsbsy, amssymb, upref,  chatterjee1}
\newcommand{\vol}{\mathrm{Vol}}

\newcommand{\ep}{\epsilon}

\begin{document}

\title[Intersection of a simplex and a sphere]
{A note about the uniform distribution on the intersection of a simplex and a sphere}
\author{Sourav Chatterjee}
\address{Courant Institute of Mathematical Sciences, New York University, 251 Mercer Street, New York, NY 10012}
\thanks{Sourav Chatterjee's research was partially supported by  NSF grants DMS-0707054 and DMS-1005312,    and a Sloan Research Fellowship}

\begin{abstract}
Uniform probability distributions on $\ell_p$ balls and spheres have been studied extensively and are known to behave like product measures in high dimensions. In this note we consider the uniform distribution on the intersection of a simplex and a sphere. Certain new and interesting features, such as  phase transitions and localization phenomena emerge.
\end{abstract}
\maketitle

\section{Introduction}\label{intro}
Take a real number number $b > 1$ and a positive integer $n$, and consider the set
\[
{\textstyle\{x = (x_1,\ldots,x_n)\in \rr^n: \sum_1^n |x_i| = n, \ \sum_1^n x_i^2 = nb\}.}
\]
This is the intersection of an $\ell_1^n$ sphere and an $\ell_2^n$ sphere in $\rr^n$. By sign symmetry, to study the above object it suffices to study
\begin{align}\label{kdef}
K := {\textstyle\{x = (x_1,\ldots,x_n)\in \rr_+^n: \sum_1^n x_i = n, \ \sum_1^n x_i^2 = nb\},}
\end{align}
where $\rr_+$ denote the set of all positive real numbers. Consider the uniform distribution on this set, defined as the limit of normalized Lebesgue measures on thin shells around this set as the thickness of the shells tend to zero. 

 Note that $b$ has to range between $1$ and $n$ for $K$ to be non-empty. We are mainly interested in $b$ fixed and $n\ra\infty$. 
Let $X = (X_1,\ldots,X_n)$ be a random vector following the uniform distribution on $K$. Let us omit the trivial case $b=1$, when all coordinates are exactly equal to $1$. The first theorem covers the range $1< b\le 2$. 
\begin{thm}\label{thm1}
Suppose $1< b\le 2$. Then there exist unique $r,s\in \rr$ such that the probability density proportional to $\exp(-rx^2-sx)$ on $[0,\infty)$ has first moment $1$ and second moment $b$. Let $Z_1,Z_2,\ldots$ be i.i.d.\ random variables following this density. The following hold:
\begin{enumerate}
\item[$(a)$] For any fixed $k$, the random vector $(X_1,\ldots,X_k)$ converges in law to $(Z_1,\ldots,Z_k)$ as $n \ra \infty$. 
\item[$(b)$] All joint moments of $(X_1,\ldots,X_k)$ converge to the corresponding moments of $(Z_1,\ldots,Z_k)$. 
\item[$(c)$] If $b < 2$, there is a constant $C$, possibly depending on $b$, such that 
\[
\lim_{n\ra \infty}\pp\bigl(\max_{1\le i\le n} X_i > C\sqrt{\log n}\bigr) = 0.
\]
\item[$(d)$] When $b=2$, part $(c)$ holds but with $\log n$ instead of $\sqrt{\log n}$. 
\end{enumerate}
\end{thm}
Note that in general $s$ can be negative. When $b=2$, $r$ and $s$ turn out to be $0$ and $1$; in other words $Z_1\sim Exp(1)$ when $b=2$. 

The next theorem describes the situation when $b > 2$. An interesting localization phenomenon occurs in this regime.  
\begin{thm}\label{thm2}
Suppose $b > 2$. Let $Z_1,Z_2,\ldots$ be i.i.d.\ $Exp(1)$ random variables. Then the following hold.
\begin{enumerate}
\item[$(a)$] For any fixed $k$, the random vector $(X_1,\ldots,X_k)$ converges in law to $(Z_1,\ldots,Z_k)$ as $n \ra \infty$.
\item[$(b)$] Convergence of moments does not happen, because $\ee(X_1^2)=b$ for any $n$ and $\ee(Z_1^2) = 2$. 
\item[$(c)$] Let $M = \max_{1\le i\le n} X_i$. Then
\[
\frac{M^2}{(b-2) n} \ra 1 \text{ in probability.}
\]
Consequently, the sum of squares of all other coordinates is roughly $2n$ with high probability. 
\item[$(d)$] Let $M_2$ be the value of the second largest coordinate. Then 
\[
\frac{M_2^2}{n} \ra 0 \text{ in probability.}
\] 
\end{enumerate}
\end{thm}

The final theorem in this note provides error bounds for the distributional convergence results in Theorems \ref{thm1} and \ref{thm2}. 
The bounds may not be sharp. 
\begin{thm}\label{thm3}
In the setting of Theorem \ref{thm1}, 
\[
\sup_{t_1,\ldots,t_k} \bigl|\pp(X_1\le t_1,\ldots,X_k \le t_k) - \pp(Z_1\le t_1,\ldots, Z_k\le t_k)\bigr|\le Ck\sqrt{\frac{\log n}{n}},
\]
where $C$ is a constant that depends only on  $b$. 
In Theorem \ref{thm2}, the bound on the right hand side becomes $Ckn^{-1/4}$. 
\end{thm}
If the condition $\sum x_i^2 = nb$ is dropped from the definition of $K$, the result is a scaled version of the standard $(n-1)$-simplex in $\rr^n$. It is a classical result in probability that the coordinates of a point chosen uniformly from this body behave like independent standard Exponential random variables in the large $n$ limit. 

On the other hand, if the condition $\sum x_i = n$ is dropped, then $K$ is just a sphere of radius $\sqrt{nb}$. Drawing uniformly from the surface of a sphere results in a vector with approximately independent Gaussian coordinates. 

A unified treatment of results of the above type was done by Diaconis and Freedman \cite{diaconisfreedman87}, which is the basic reference for the literature in this area till~1987. 

In recent times, attention has shifted to the study of the $\ell_p^n$ balls and spheres, that is, sets where $\sum|x_i|^p$ is bounded by or equal to a constant. The distribution of low dimensional projections for $\ell_p^n$ balls was obtained by Naor and Romik \cite{naorromik03}, who showed that the coordinates behave like i.i.d.\ random variables with density proportional to $e^{-|x|^p}$. An extensive investigation of the probabilistic structure of $\ell_p^n$ balls was done by Barthe et.\ al.\ \cite{bartheetal05}.

The volumes of intersections of $\ell_p^n$ balls (not spheres) have been previously investigated, in response to a question raised by Vitali Milman, in a series of papers by Schechtman and Zinn \cite{schechtmanzinn90}, Schechtman and Schmuckenschl\"ager \cite{ss91} and Schmuckenschl\"ager \cite{schmuckenschlager98, schmuckenschlager01}. They do not, however, study the behavior of uniformly chosen random points from these sets.



The main motivation for this paper, however, comes from certain observations in the physics literature. The phase transition that is rigorously established by Theorems \ref{thm1} and \ref{thm2} was previously identified in a physics paper of Rumpf \cite{rumpf04}. Rumpf's investigation was motivated by a desire to understand localization of energy in discrete nonlinear Schr\"odinger equations. A rigorous mathematical program of investigating the localization in nonlinear Schr\"odinger equations, based on the techniques developed in this manuscript, has been developed in \cite{cha12, ck10}.  In a different direction, applications of these techniques to localization random geometric graphs have been worked out in \cite{chaharel}.

\section{Preliminaries}\label{prelim}

For a vector $x = (x_1,\ldots,x_n)\in \rr^n$, define
\[
\mu(x) := \frac{1}{n}\sum_{i=1}^n x_i, \ \ \ \mu_2(x) := \frac{1}{n}\sum_{i=1}^n x_i^2 = \frac{\|x\|^2}{n}.
\]
Also define
\[
\sigma(x) := \sqrt{\mu_2(x)-\mu^2(x)}, \ \ \ m(x):= \min_{1\le i\le n} x_i.
\]
Fix $b > 1$ as in Section \ref{intro} and let $b' = \sqrt{b-1}$. Let $\rr_+$ be the set of positive real numbers. Note that according to the definition \eqref{kdef}, 
\begin{equation*}
\begin{split}
K &= \{x\in \rr_+^n : \mu(x)=1, \ \mu_2(x)=b\}\\
&= \{x\in \rr_+^n : \mu(x)=1, \ \sigma(x)=b'\}. 
\end{split}
\end{equation*}
The notation introduced above will be used without explicit reference in the rest of the manuscript. 


Recall that the uniform distribution on the unit sphere in any dimension is equivalently defined as the unique probability measure that is invariant under rotations (i.e.\ the action of orthogonal matrices). 
For each $a, d\in \rr$, $c >0$, define
\begin{align*}
&S(a,c) := \{x\in \rr^n: \mu(x)=a, \ \sigma(x)=c\}, \\ &S(a,c,d) := \{x\in S(a,c): m(x) > d\}. 
\end{align*}
Note that $S(0,1)$ is a sphere in the $n-1$ dimensional hyperplane $\{x: \mu(x)=0\}$, centered at the origin. This hyperplane can be obtained as the image of $\rr^{n-1} = \{x\in \rr^n: x_n =0\}$ under any rotation in $\rr^n$ that takes the point $(0,0,\ldots,0,\sqrt{n})$ to $(1,1,\ldots,1) =: {\bf 1}$. 

At the risk of coming across as too pedantic, I will now define the uniform probability distribution on $K$. Consider the map $\phi:\rr^n \ra S(0,1)$ defined as $\phi(x) := 0$ if $x = \alpha \bf{1}$ for some scalar $\alpha$, and 
\begin{equation}\label{phidef}
\phi(x) := \frac{1}{\sigma(x)}(x-\mu(x){\bf1}) \ \ \text{otherwise}. 
\end{equation}
Let $Z$ be an $n$-dimensional standard Gaussian random vector. Let $A$ be an orthogonal matrix satisfying $A{\bf1}=\bf{1}$. (This is the set of all rotations preserving $S(0,1)$.) Then $AZ$ is again standard Gaussian. Note that 
\[
\mu(AZ) = \frac{1}{n}{\bf1}^T AZ = \frac{1}{n}{\bf1}^TZ = \mu(Z).
\]
Since $\|AZ\|=\|Z\|$, this implies that $\sigma(AZ) = \sigma(Z)$. Since $\sigma(Z) >0$ almost surely, the above steps can be combined to give 
\begin{equation}\label{az}
A\phi(Z) = \frac{1}{\sigma(AZ)} (AZ - \mu(AZ)) = \phi(AZ) \stackrel{d}{=} \phi(Z).  
\end{equation}
Since this holds for every rotation $A$ of the sphere $S(0,1)$, $\phi(Z)$ is uniformly distributed on $S(0,1)$.

Now suppose $S(0,1,d) \ne \emptyset$. Then there exists $x\in S(0,1)$ such that $m(x) > d$. Since $m(x) = m(\phi(x))$  for this $x$ and $m\circ \phi$ is a continuous map in a neighborhood of $S(0,1)$, there exists a ball $B$ of positive radius centered at $x$ such that $m(\phi(y)) > d$ for all $y\in B$. Since $\pp(Z\in B) > 0$ this shows that $\pp(m(\phi(Z)) > d) > 0$, and hence the uniform distribution on $S(0,1)$ puts positive mass on $S(0,1,d)$. Therefore the uniform distribution on $S(0,1,d)$ is simply the restriction of the uniform distribution on $S(0,1)$ to this set. Since\begin{equation}\label{ks}
K = b'S(0,1, -1/b') + {\bf1},
\end{equation}
this gives an alternative characterization of the uniform distribution on $K$ that will be convenient for our purposes. 

\section{From thin sets to thick sets}
In this section, we show how to deduce results about $K$ from a slight `positively tilted' thickening of $K$, that we call $K^\ep$. 
For any $\epsilon >0$, let
\[
K^\epsilon := \{x\in \rr_+^n : \epsilon< \mu(x)-1<2\epsilon, \ \epsilon < \mu_2(x)-b < b\epsilon\}. 
\]
Clearly, $K^\ep$ has nonzero volume whenever it is non-empty, and therefore the uniform distribution on $K^\epsilon$ is naturally defined as restriction of the Lebesgue measure, normalized to have mass $1$. 

Define a map $\psi :\rr^n \ra \rr^n$ as 
\begin{equation}\label{psidef}
\psi(x) = b'\phi(x) + {\bf 1}, 
\end{equation}
where $\phi$ is the map defined in \eqref{phidef} and $b' = \sqrt{b-1}$. 
\begin{prop}\label{imp}
Let $K$, $K^\epsilon$ and $\psi$ be defined as above, and suppose $K^\epsilon$ is non-empty. Let $X$ be a random vector that is uniformly distributed on $K$, and let $X^\epsilon$ be a random vector that is uniformly distributed on $K^\epsilon$. Then for any $f:\rr_+^n \ra [0,\infty) $ and any $\epsilon\in (0,c(b))$, 
\[
\ee f(X)\le \frac{\ee f(\psi(X^\epsilon))}{\pp(m(X^\epsilon) > C(b)\epsilon)},
\]
where $c(b) (< 1/2)$ and $C(b)$ are positive constants that depend only on the value of~$b$. The right hand side is interpreted as infinity if the denominator is zero. 
\end{prop}
\begin{proof}
Let $\ep$ be so small that $\ep < 1/2$ and $7\ep <  b-1$. Define
\begin{align*}
&\hat{K}^\epsilon := \{x\in \rr^n: \epsilon < \mu(x) - 1< 2\epsilon, \ \epsilon < \mu_2(x)-b< b\epsilon\},
\end{align*}
so that $K^\epsilon = \{x\in \hat{K}^\epsilon: m(x) >0\}$.  
Note that
\begin{equation}\label{sigmaup}
\begin{split}
\sigma(x) &= (\mu_2(x)-\mu^2(x))^{1/2}\\
&< (b+b\ep - (1+\ep)^2)^{1/2}\\
&\le ((b-1)(1+\ep))^{1/2} \le b'(1+\ep),
\end{split} 
\end{equation}
and
\begin{equation}\label{sigmadown}
\begin{split}
\sigma(x) &> (b+\ep - (1+2\ep)^2)^{1/2}\\
&= (b- 1-3\ep-4\ep^2)^{1/2}\\
&> ((b-1)(1-7\ep/(b-1)))^{1/2} > b'(1-7\ep/(b-1)). 
\end{split} 
\end{equation}
The last inequality shows that, in particular, $\sigma(x) >0$ and hence $x$ cannot belong to the diagonal line. 
Let $l$ be the linear transformation
\[
l(x) := b'x + {\bf 1}. 
\]
Let $d := 1/b'$, so that $l^{-1}(x) = d(x-{\bf 1})$. As pointed out before in \eqref{ks}, $S(0,1,-d) = l^{-1}(K)$. Define 
\begin{align*}
&\hat{S}^\epsilon := l^{-1}(\hat{K}^\epsilon), \ \  S^\epsilon := \{x\in \hat{S}^\epsilon: m(x) > -d\} = l^{-1}(K^\ep).
\end{align*}
Let $Y^\epsilon$ be uniformly distributed on $\hat{S}^\ep$. Since $\hat{K}^\ep$ does not intersect the diagonal line, it follows that  $\hat{S}^\ep$ does not intersect the diagonal line either. We claim that $\phi(Y^\epsilon)$ is uniformly distributed on $S(0,1)$. To see this, let $A$ be an orthogonal matrix such that $A{\bf 1} = {\bf1}$. As argued to derive \eqref{az}, we see that
\begin{equation}\label{ay}
A\phi(Y^\epsilon) = \phi(AY^\epsilon)\ \ \text{a.s.} 
\end{equation}
Again, as argued before, $\mu(Ax) = \mu(x)$ and $\mu_2(Ax)=\mu_2(x)$ for any $x$ outside the diagonal line, and therefore, $A$ maps $\hat{K}^\ep$ onto itself. By the property that $A{\bf 1} = {\bf 1}$ it follows that $A$ and $l^{-1}$ commute, and thus $A$ maps $\hat{S}^\ep$ onto itself. Since $A$ is a linear map, this shows that $AY^\epsilon$ is uniformly distributed on $\hat{S}^\ep$. Combined with \eqref{ay}, this proves the claim that $\phi(Y^\epsilon)$ is uniformly distributed on~$S(0,1)$. 

Now, clearly, 
\[
m(\phi(Y^\epsilon)) = \frac{1}{\sigma(Y^\epsilon)} (m(Y^\epsilon) - \mu(Y^\epsilon)). 
\]
Thus, $m(\phi(Y^\epsilon)) > - d$ if and only if 
\[
m(Y^\epsilon) > \mu(Y^\epsilon) - d\sigma(Y^\epsilon). 
\]
Therefore, if $U$ is distributed uniformly on $S(0,1,-d)$, then for any measurable $h:\rr^n\ra [0,\infty)$,
\begin{equation}\label{hu}
\begin{split}
\ee h(U) &= \ee(h(\phi(Y^\epsilon)) \mid m(\phi(Y^\epsilon)) > -d) \\
&= \ee(h(\phi(Y^\epsilon)) \mid m(Y^\epsilon) > \mu(Y^\epsilon) - d\sigma(Y^\epsilon)). 
\end{split}
\end{equation}
Now take any $y\in \hat{S}^\ep$ and let $x = l(y)$. Then $x\in \hat{K}^\ep$, and therefore by \eqref{sigmaup},
\begin{align*}
\mu(y) - d\sigma(y) &= d(\mu(x) - 1) - d^2\sigma(x)\\
&> d\epsilon - d(1+\ep) = -d.
\end{align*}
Thus, the event $m(Y^\epsilon) > \mu(Y^\epsilon) - d\sigma(Y^\epsilon)$ implies $m(Y^\epsilon) > -d$. Let $Z^\epsilon$ be uniformly distributed on $S^\epsilon$. Then the law of $Z^\epsilon$ is the same as that of $Y^\epsilon$ conditioned on the event $m(Y^\epsilon) > -d$. Combined with the previous step and \eqref{hu}, we get 
\begin{align*}
\ee h(U) &=\ee(h(\phi(Y^\epsilon)) \mid m(Y^\epsilon) > \mu(Y^\epsilon) - d\sigma(Y^\epsilon)) \\
&= \ee(h(\phi(Y^\epsilon)) \mid m(Y^\epsilon) > \mu(Y^\epsilon) - d\sigma(Y^\epsilon), \ m(Y^\epsilon) > -d)\\
&= \ee(h(\phi(Z^\epsilon)) \mid m(Z^\epsilon) > \mu(Z^\epsilon) - d\sigma(Z^\epsilon)).
\end{align*}
Since $h$ is a non-negative function, this implies that 
\[
\ee h(U) \le 
\frac{\ee h(\phi(Z^\epsilon))}{\pp(m(Z^\epsilon) > \mu(Z^\epsilon) - d\sigma(Z^\epsilon))}. 
\]
(If the denominator is zero we interpret the right hand side as infinity.)
However, for any $y = l(x)\in \hat{S}^\epsilon$,  \eqref{sigmadown} gives 
\begin{align*}
\mu(y) - d\sigma(y) &= d(\mu(x)-1) - d^2\sigma(x)\\
&\le 2d\epsilon - d(1-7d^2\epsilon)= -d + (2d+7d^3)\epsilon. 
\end{align*}
Thus, 
\begin{equation}\label{bd1}
\ee h(U) \le \frac{\ee h(\phi(Z^\epsilon))}{\pp(m(Z^\epsilon) > -d + (2d+7d^3)\epsilon)}. 
\end{equation}
Since $X$ has the same law as $l(U)$, we get
\begin{align*}
\ee f(X) &= \ee(f\circ l (U))
\le \frac{\ee(f\circ l\circ \phi(Z^\epsilon))}{\pp(m(Z^\epsilon) > -d + (2d+7d^3)\epsilon)}.
\end{align*}
Since $\psi = l\circ \phi$ and $\phi = \phi\circ l$, this gives 
\begin{align*}
\ee f(X) &\le \frac{\ee(f\circ \psi(l(Z^\epsilon)))}{\pp(m(Z^\epsilon) > -d + (2d + 7d^3)\epsilon)}.
\end{align*}
Again, since $l$ is a linear bijection between $S^\ep$ and~$K^\epsilon$, $l(Z^\epsilon)$ is uniformly distributed on $K^\epsilon$.  Thus, 
\[
\ee(f\circ \psi(l(Z^\epsilon)) ) = \ee(f\circ \psi(X^\epsilon)).
\]
Finally, note that $m(l(Z^\epsilon)) = d^{-1}m(Z^\epsilon) + 1$, and hence
\[
\pp(m(Z^\epsilon) > -d + (2d+7d^3)\epsilon) = \pp(m(X^\epsilon) > (2+7d^2)\epsilon).
\]
This completes the proof. 
\end{proof}

\begin{prop}\label{propbd1}
Suppose $K^\ep$ is non-empty. Let $c(b)$ and $C(b)$ be as in Proposition \ref{imp}, and suppose $\ep\in (0,c(b))$.  Suppose $g:\rr^n\ra \rr$ is a function such that there is a constant $L$, such that for all $x,y\in \rr^n$,
\[
|g(x)-g(y)|\le L\max_{1\le i\le n} |x_i - y_i|. 
\]
Then for any $a,t\in \rr$, 
\begin{align*}
\pp(|g(X)-a| > t) \le  \frac{\pp(|g(X^\epsilon)-a| > t- C_3(b)L\epsilon n)}{\pp(m(X^\epsilon) > C(b)\epsilon)},
\end{align*}
where $C_3(b)$ is another constant depending only on $b$.
\end{prop}
\begin{proof}
Take any $x\in K^\epsilon$ and let $y = \psi(x)$. Then for any $i$, we can use the definition of $K^\ep$ and the inequalities \eqref{sigmaup} and \eqref{sigmadown} to conclude that
\begin{align*}
|x_i - y_i| &= \biggl|\biggl(1-\frac{b'}{\sigma(x)}\biggr)x_i + \frac{b'\mu(x)}{\sigma(x)} - 1\biggr|\\
&\le \frac{|(b'-\sigma(x))x_i|}{\sigma(x)} + \frac{b'|\mu(x)-1| + |b'-\sigma(x)|}{\sigma(x)}\\
&\le C_2(b)\ep (1+x_i), 
\end{align*}
where $C_2(b)$ is a constant depending only on $b$. 
Since $c(b) < 1/2$, we have   
\[
x_i\le\sum_j x_j \le n(1+2\epsilon)\le 2n.
\]
Thus, taking $C_3(b) = 3C_2(b)$, we have 
\[
\max_i |x_i-y_i|\le C_2(b)\epsilon (2n+1)\le C_3(b)\epsilon n. 
\]
Therefore for any $x\in K^\epsilon$,
\begin{align*}
|g(x)-g(\psi(x))| &\le C_3(b)L\epsilon n.  
\end{align*}
In particular, the event $|g(\psi(x)) - a| > t$ implies 
\[
|g(x)-a| > t - C_3(b)L\epsilon n. 
\]
Taking $f(x) := 1_{\{|g(x)-a|>t\}}$, we get by Proposition \ref{imp} that
\begin{align*}
\pp(|g(X)-a| > t) &\le \frac{\pp(|g(\psi(X^\epsilon))-a| > t)}{\pp(m(X^\epsilon) > C(b)\epsilon)}\\
&\le \frac{\pp(|g(X^\epsilon)-a| > t- C_3(b)L\epsilon n)}{\pp(m(X^\epsilon) > C(b)\epsilon)}. 
\end{align*}
This completes the proof.
\end{proof}

\section{From thick sets to conditional distributions}\label{thickcond}
In this section, we show that the uniform distribution on $K^\ep$ be can approximated by the distribution of a random vector with independent coordinates conditioned to be in $K^\ep$. 

For each $(r,s)\in (\rr_+\times \rr)\cup(\{0\}\times\rr_+)$, let $G_{r,s}$ be the probability distribution on $\rr_+$ with probability density proportional to $\exp(-rx^2-sx)$ on $(0,\infty)$. Note that if $(r,s)\not \in (\rr_+\times \rr)\cup(\{0\}\times\rr_+)$, $\exp(-rx^2 - sx)$ is not integrable on $(0,\infty)$. Henceforth, whenever we say `for any $r,s$', we will mean `for any $(r,s)$ in this admissible region'. 

In the following $G_{r,s}^{\otimes n}$ will denote the $n$-fold product of $G_{r,s}$ as a probability measure on $\rr^n$. 
\begin{lmm}\label{bdlmm}
Let $c(b)$, $C(b)$ and $C_3(b)$ be as in Proposition \ref{propbd1}. Take any $\epsilon \in (0,c(b))$ such that $K^\epsilon$ is non-empty. Suppose $Y \sim G^{\otimes n}_{r,s}$ for some $r,s$. Then for any function $f:K^\epsilon \ra [0,\infty)$ we have
\[
e^{-B\epsilon n}\ee f(X^\epsilon) \le  \ee( f(Y) \mid Y\in K^\epsilon)\le e^{B\epsilon n}\ee f(X^\epsilon), 
\]
where $B  = 2br + 4|s|$. 
\end{lmm}
\begin{proof}
Recall that for $x\in K^\epsilon$, 
\[
|\mu(x)-1|\le 2\epsilon, \ \ |\mu_2(x)-b|\le b\ep.
\]
Therefore, if we set $B = 2br + 4|s|$, it follows that 
\begin{align*}
\ee(f(Y)\mid Y\in K^\epsilon) &= \frac{\int_{K^\epsilon}f(x) e^{-rn\mu_2(x)-sn\mu(x)} dx }{\int_{K^\epsilon}e^{-rn\mu_2(x)- sn\mu(x)} dx}\\
&\ge e^{-B\epsilon n}\frac{\int_{K^\epsilon}f(x) e^{-rnb-sn} dx }{\int_{K^\epsilon}e^{-rnb- sn} dx}\\
&= e^{-B\epsilon n}\frac{\int_{K^\epsilon}f(x) dx }{\int_{K^\epsilon} dx} = e^{-B\epsilon n} \ee f(X^\epsilon). 
\end{align*}
Similarly, we get the other bound. 
\end{proof}

\begin{prop}\label{propbd2}
Let $c(b)$, $C(b)$ and $C_3(b)$ be as in Proposition \ref{propbd1}. 
Take $\epsilon \in (0,c(b))$ such that $K^\epsilon$ is non-empty. Suppose $g$ is a function as in Proposition \ref{propbd1}, and $Y\sim G^{\otimes n}_{r,s}$ for some $r,s$. Then for  any $a,t\in \rr$, we have
\begin{align*}
\pp(|g(X)-a| > t) \le  e^{2B\epsilon n}\frac{\pp(|g(Y)-a| > t- C_3(b)L\epsilon n, \; Y\in K^\ep)}{\pp(m(Y) > C(b)\epsilon, \; Y\in K^\epsilon)},
\end{align*}
where $B  = 2br + 4|s|$. 
\end{prop}
\begin{proof}
By Lemma \ref{bdlmm} we see that
\begin{align*}
&\pp(|g(X^\epsilon)-a|> t-C_3(b)L\epsilon n ) 
\\
&\le e^{B\epsilon n} \pp(|g(Y)-a|> t-C_3(b)L\epsilon n \mid Y\in K^\epsilon) 
\end{align*}
and
\[
\pp(m(X^\epsilon) > C(b)\epsilon) \ge e^{-B\epsilon n} \pp(m(Y) > C(b)\epsilon \mid Y\in K^\epsilon).
\]
Using these bounds in Proposition \ref{propbd1}, we get
\begin{align*}
\pp(|g(X)-a| > t) &\le  e^{2B\epsilon n}\frac{\pp(|g(Y)-a| > t- C_3(b)L\epsilon n \mid Y\in K^\epsilon)}{\pp(m(Y) > C(b)\epsilon\mid Y\in K^\epsilon)}\\
&= e^{2B\epsilon n}\frac{\pp(|g(Y)-a| > t- C_3(b)L\epsilon n, \; Y\in K^\epsilon)}{\pp(m(Y) > C(b)\epsilon, \; Y\in K^\epsilon)}.
\end{align*}
This completes the proof. 
\end{proof}

\section{A local limit theorem}
In this section we derive some basic properties of the probability distribution $G_{r,s}$ defined in the previous section. Fix $r,s$, and let $Y_1, Y_2,\ldots i.i.d.\sim G_{r,s}$. Suppose $\ee(Y_1)=1$ and let $\beta:= \ee(Y_1^2)$.
\begin{lmm}\label{bdddens}
The pair $(Y_1+Y_2+Y_3, \; Y_1^2+Y_2^2+Y_3^2)$ has a bounded density in $\rr^2$. 
\end{lmm}
\begin{proof}
In this proof, $C$ will denote any positive constant that may depend on $b$, $r$ or $s$, but no other parameters. The value of $C$ may change from line to line. Fix any $u,v\in \rr$  and $\delta > 0$.  
Let 
\[
A := \{(y_1,y_2,y_3): |y_1+y_2+y_3 - u| < \delta, \ |y_1^2+y_2^2 + y_3^2 - v| < \delta\}.
\]
Note that the probability density of $(Y_1,Y_2, Y_3)$ is uniformly bounded. Therefore 
\[
\pp((Y_1,Y_2, Y_3)\in A)\le C \,\vol(A)\,.
\]
For each $x\in \rr$, let
\[
A_x := \{(y_1,y_2,y_3): y_1+y_2+y_3=x\} \cap A\,.
\]
Clearly, $A_x$ is either empty, or is a two-dimensional annulus whose area is bounded by $C\delta$. It follows that $\vol(A)\le C\delta^2$. From here, it is easy to argue that the distribution of $(Y_1+Y_2+Y_3, Y_1^2+Y_2^2+Y_3^2)$ is absolutely continuous with respect to the Lebesgue measure on~$\rr^2$, with uniformly bounded density (see Theorem 7.14 in \cite{rudin87}). 
\end{proof}

\begin{lmm}\label{local0}
The sequence $V_n:= n^{-1/2}(\sum_1^n (Y_i-1),\; \sum_1^n (Y_i^2-\beta))$ satisfies a uniform local limit theorem, meaning that there is a non-degenerate Gaussian density $\rho$ on $\rr^2$ such that if $\rho_n$ is the probability density of $V_n$, then 
\[
\lim_{n\ra \infty} \sup_{(x,y)\in \rr^2} |\rho_n(x,y)-\rho(x,y)| = 0.
\]
\end{lmm}
\begin{proof}
The result follows directly from Lemma \ref{bdddens} and the classical uniform local limit theorem, e.g.\ Theorem 19.1 in~\cite{bhattacharya86}. The non-degeneracy holds because the covariance matrix of $(Y_1,Y_1^2)$ is obviously non-singular. 
\end{proof}

\begin{lmm}\label{local}
Suppose for each $n$ we have real numbers $a_n \le b_n$, $a_n'\le b_n'$ such that there exist $x_0,y_0\in \rr$, with
\begin{align*}
&\lim_{n\ra \infty} \sqrt{n}(a_n - 1) = \lim_{n\ra \infty} \sqrt{n}(b_n - 1) = x_0,\\
&\lim_{n\ra \infty} \sqrt{n}(a_n' - \beta) = \lim_{n\ra \infty} \sqrt{n}(b_n' - \beta) = y_0.
\end{align*}
Then 
\begin{align*}
\lim_{n\ra \infty} \frac{\pp(a_n \le \frac{1}{n}\sum_{i=1}^n Y_i \le b_n, \ a_n' \le \frac{1}{n}\sum_{i=1}^n Y_i^2 \le b_n')}{n(b_n-a_n)(b_n'-a_n')} = \rho(x_0,y_0),
\end{align*}
where $\rho$ is as in Lemma \ref{local0}. 
\end{lmm}
\begin{proof}
Let $\rho_n$ be as in Lemma \ref{local0}. Let 
\begin{align*}
u_n := \sqrt{n}(a_n - 1), \ \ v_n :=  \sqrt{n}(b_n - 1),\\
u_n' := \sqrt{n}(a_n' - \beta), \ \ v_n' :=  \sqrt{n}(b_n' - \beta).
\end{align*}
Then
\begin{align*}
&{\textstyle \pp(a_n \le \frac{1}{n}\sum_{i=1}^n Y_i \le b_n, \ a_n' \le \frac{1}{n}\sum_{i=1}^n Y_i^2 \le b_n')} \\
&= \int_{u_n}^{v_n} \int_{u_n'}^{v_n'} \rho_n(x,y) dy dx. 
\end{align*}
Let
\[
\delta_n := \sup_{(x,y)\in \rr^2} |\rho_n(x,y)-\rho(x,y)|. 
\]
and 
\[
\tau_n := \sup_{u_n\le x\le v_n, \ u_n'\le y\le v_n'} |\rho(x,y) - \rho(x_0, y_0)|. 
\]
Then $\delta_n \ra 0$ by Lemma \ref{local0}, and $\tau_n\ra 0$ due to continuity of $\rho$. Finally, observe that
\begin{align*}
&\biggl| \int_{u_n}^{v_n} \int_{u_n'}^{v_n'} \rho(x,y) dy dx - n(b_n-a_n)(b_n'-a_n') \rho(x_0,y_0)\biggr|\\
&= \biggl| \int_{u_n}^{v_n} \int_{u_n'}^{v_n'} (\rho(x,y) - \rho(x_0,y_0)) dy dx\biggr|\le \tau_n n(b_n-a_n)(b_n'-a_n'),
\end{align*}
and 
\begin{align*}
\biggl| \int_{u_n}^{v_n} \int_{u_n'}^{v_n'} (\rho_n(x,y) -\rho(x,y)) dy dx\biggr|&\le  \delta_n n(b_n-a_n)(b_n'-a_n'),
\end{align*}
This completes the proof. 
\end{proof}

\section{Proof of Theorem \ref{thm1}}
In this section we consider the situation $1< b\le 2$. First, we need to show the existence of $r,s$ such that the probability distribution $G_{r,s}$ has first moment $1$ and second moment $b$. The uniqueness of $r,s$ will follow automatically from the distributional convergence result for $X_1$. 
\begin{prop}
If $1< b \le 2$, there exist $r,s\in \rr$ such that the probability distribution $G_{r,s}$ defined in Section \ref{thickcond} has mean $1$ and second moment $b$. 
\end{prop}
\begin{proof}
Clearly, if $W\sim G_{r,s}$ then for any $\alpha >0$, $\alpha W\sim G_{r',s'}$ for some other $r',s'$. Thus, it suffices to show that for any $b\in (1,2]$, there exists $r,s$ such if $W\sim G_{r,s}$, then
\[
\theta(r,s) := \frac{\ee(W^2)}{(\ee(W))^2} = b. 
\]
It is easy to see that $\theta$ is a continuous function of $r,s$. Since $G_{0,1}$ is just the $Exp(1)$ distribution, $\theta(0,1) = 2$. For each $r>0$, let $W_r\sim G_{r,1}$. Let $Z_r := \sqrt{r}W_r$. Then the density of $Z_r$ on $[0,\infty)$ is proportional to $\exp(-z^2 - z/\sqrt{r})$. It is easy to argue from here that as $r\ra \infty$, $Z_r$ converges in law to $Z$, which has density proportional to $\exp(-z^2)$. Moreover, the moments of $Z_r$ converge to those of $Z$. 

Thus, by the  intermediate value theorem for continuous functions, we see that as $r$ ranges between $0$ and $\infty$, $\theta(r,1)$ takes all values between $\theta_0 := \ee(Z^2)/(\ee(Z))^2$ and~$2$. (It is easily verified that $1<\theta_0 \le 2$.) 
Next, for $0\le u< 1$, let $V_u$ follow the density 
\[
\rho(v) \propto\exp\biggl(-\frac{(v-u)^2}{1-u}\biggr), \ \ v\ge 0.
\]
In other words,  $V_u \sim G_{1/(1-u), -2u/(1-u)}$. Note that $V_0$ has the same distribution as $Z$, and as $u\ra 1$, the law of $V_u$ tends to the point mass at~$1$. Convergence of moments is again easy to prove. Therefore, again, by the intermediate value theorem we see that $\theta(1/(1-u), -2u(1-u))$ ranges over all values between $\theta_0$ and $1$ as $u$ varies between $0$ and $1$. This completes the proof. 
\end{proof}

\begin{proof}[Proof of Theorem \ref{thm1}, part $(a)$]
Choose $r,s$ such that $G_{r,s}$ has first moment $1$ and second moment $b$. In this proof, $C$ will always denote any positive constant that may depend only on $b$, $r$ or $s$ and no other parameter. (A~priori, we do not yet know that $r,s$ are uniquely determined by $b$, so we treat them as independent parameters.) Let $c(b)$, $C(b)$ and $C_3(b)$ be as in Proposition~\ref{propbd2}. 

Fix $\epsilon = n^{-10}$. It will be evident from the proof that the exponent $10$ is not of any consequence; any sufficiently large exponent would do. By Lemma~\ref{local}, we see that for sufficiently large $n$, 
\begin{equation}\label{kbd}
C^{-1}n\ep^2 \le \pp(Y\in K^\ep) \le Cn\ep^2.
\end{equation}
(Note that this proves, in particular, that $K^\ep$ is non-empty.)
Now, if $Y_1 \le C(b) \ep$ and $Y\in K^\ep$, then 
\begin{align}\label{y21}
\frac{1}{n}\sum_{i=2}^n Y_i = \mu(Y) - \frac{Y_1}{n} \in (1+\ep-n^{-1}C(b)\ep, \; 1+2\ep),
\end{align}
and 
\begin{align}\label{y22}
\frac{1}{n}\sum_{i=2}^n Y_i^2 = \mu_2(Y)-\frac{Y_1^2}{n} \in (b+\ep - n^{-1}C(b)^2\ep^2, b+\ep b). 
\end{align}
Let $E$ be the event that the two events \eqref{y21} and \eqref{y22} happen. By Lemma~\ref{local}, we see that
\[
\pp(E) \le Cn\ep^2. 
\]
Moreover, the event $E$ is independent of the event $\{Y_1 \le C(b)\ep\}$. Thus,
\begin{align*}
\pp(Y_1 \le C(b) \ep,\; Y\in K^\ep) &\le \pp(\{Y_1\le C(b)\ep\}\cap E) \\
&= \pp(Y_1\le C(b)\ep)\pp(E) \le Cn \ep^3. 
\end{align*}
Combining with \eqref{kbd}, and observing that $n^2\ep^3 \ll n\ep^2$, we get that for sufficiently large $n$, 
\begin{equation}\label{ylow}
\begin{split}
&\pp(m(Y) > C(b)\ep, \; Y\in K^\ep)\\
 &\ge \pp(Y\in K^\ep) - n\pp(Y_1\le C(b)\ep, \; Y\in K^\ep) \\
&\ge C^{-1}n\ep^2 - Cn^2\ep^3\ge  C^{-1} n\ep^2. 
\end{split}
\end{equation}
Next, let $h:\rr \ra \rr$ be a function satisfying $|h(x)|\le 1$ and $|h(x)-h(y)|\le L|x-y|$ for all $x,y\in \rr$, where $L$ is some positive constant. Define $g:\rr^n \ra \rr$ as 
\[
g(x) := \frac{1}{n}\sum_{i=1}^n h(x_i). 
\]
Note that 
\[
|g(x)-g(y)| \le L \max_i |x_i-y_i|.
\]
Let $a:=\ee h(Y_1)$. Then by Hoeffding's tail inequality for sums of independent bounded random variables \cite{hoeffding63}, we have that for any $t > 0$,
\[
\pp(|g(Y)-a|>t)\le 2e^{-nt^2/2}. 
\]
Therefore it follows from Proposition \ref{propbd2} and \eqref{ylow} that for all $t> C_3(b)L\ep n$, 
\begin{equation}\label{tails}
\pp(|g(X)-a| > t) \le C n^{-1}\ep^{-2}e^{-n(t-C_3(b)L\ep n)^2/2}. 
\end{equation}
The tail bound decays rapidly in the regime $t > C(n^{-1}\log n)^{1/2}$; from this it is easy to deduce that
\[
\ee|g(X)-a| \le C\sqrt{\frac{\log n}{n}}.
\]
By Jensen's inequality and symmetry, we have
\[
\ee|g(X)-a| \ge |\ee h(X_1) - a| = |\ee h(X_1)- \ee h(Y_1)|.
\]
This shows the convergence of in law for $X_1$. To show joint convergence for $X_1,\ldots, X_k$, we proceed as follows. Instead of a single function $h$, consider $k$ functions $h_1,\ldots,h_k$, each satisfying $|h_i(x)|\le 1$ and $|h_i(x)-h_i(y)|\le L|x-y|$. Define $g_1,\ldots, g_k$ and $a_1,\ldots, a_k$ accordingly. Then $|g_i|\le 1$, and therefore by a simple telescoping argument 
\[
\ee\biggl|\prod_{i=1}^k g_i(X) - \prod_{i=1}^k a_i\biggr|\le k\max_i \ee|g_i(X)-a_i|\le  Ck\sqrt{\frac{\log n}{n}}. 
\]
Now, putting $A := \prod a_i$ and using Jensen's inequality, we see that
\begin{align*}
\ee\biggl|\prod_{i=1}^k g_i(X) - \prod_{i=1}^k a_i\biggr| &\ge \biggl|\frac{1}{n^k}\sum_{1\le i_1,\ldots,i_k\le n} (\ee(h_{1}(X_{i_1})h_{2}(X_{i_2})\cdots h_k(X_{i_k})) - A)\biggr|\\
&= |\ee(h_1(X_1)\cdots h_k(X_k)) - A| + O(1/n). 
\end{align*}
This shows the joint convergence of $X_1,\ldots, X_k$ and completes the proof of part~$(a)$. 
Finally, as we noted before, the distributional convergence automatically proves the uniqueness of $r,s$. 
\end{proof}

\begin{proof}[Proof of parts $(c)$ and $(d)$]
Let $g(x)=\max_i x_i$. Then 
\[
|g(x)-g(y)|\le \max_i |x_i-y_i|.
\]
When $b < 2$, we must have $r >0$. In this situation,  it is not difficult to conclude that 
\[
\pp(g(Y) > t) \le ne^{-t^2/C}.
\]
Thus by Proposition \ref{propbd2} and \eqref{ylow} it follows that for all $t > C_3(b)\ep n$, 
\[
\pp(g(X) > t) \le C \ep^{-2}e^{-(t-C_3(b)\ep n)^2/C}. 
\]
This shows that there exists a constant $C$ depending only on $b$ such that 
\[
\lim_{n\ra\infty}\pp(\max_{1\le i\le n} X_i > C\sqrt{\log n}) = 0.
\]
When $b =2$, we have $r = 0$ and $s=1$. The argument is exactly the same, except that the tail bound is $e^{-t/C}$ instead of $e^{-t^2/C}$, which gives the $\log n$ instead of $\sqrt{\log n}$. 
\end{proof}

\begin{proof}[Proof of part $(b)$]
Let us first prove for $k=1$. Suppose we want to prove the convergence of the $p$th moment. Fix $n$. For $x>0$, let 
\[
h(x):= \min\{x^p, (\log n)^{2p}\}.
\]
Let us compute a Lipschitz constant for $h$. If $x > (\log n)^2$ and $y>(\log n)^2$, then $h(x)-h(y)=0$. If $x\le (\log n)^2$ and $y \le (\log n)^2$, then
\begin{align*}
|h(x)-h(y)| &= |x^p - y^p|\\
&= |x-y||x^{p-1} + x^{p-2} y +\cdots + y^{p-1}|\\
&\le p(\log n)^{2p-2} |x-y|. 
\end{align*}
Finally, if $x\le (\log n)^2$ but $y > (\log n)^2$, the
\begin{align*}
|h(x)-h(y)| &= |h(x)-h((\log n)^2)| \\
&\le p(\log n)^{2p-2}|x-(\log n)^2| \le p(\log n)^{2p-2} |x-y|. 
\end{align*}
Thus, we can take $L = p(\log n)^{2p-2}$ and proceed as in the proof of part $(a)$ to get \eqref{tails}. This proves that $\ee\min\{X_1^p , (\log n)^2\} \ra \ee(Z_1^p)$. Now, from the proof of part $(c)$ and the fact that $0\le X_1 \le n$,  we see that when $1< b \le 2$,
\begin{align*}
|\ee\min\{X_1^p , (\log n)^2\} - \ee(X_1^p)| &\le n^p \pp(X_1 > (\log n)^2) \\
&\le n^p \ep^{-2}e^{-(\log n)^2/C} \ra 0. 
\end{align*}
This completes the proof for $k=1$. For $k >1$, and a monomial like $x_1^{p_1}\cdots x_k^{p_k}$ we proceed as in part $(a)$ by defining 
\[
g_i(x) := \frac{1}{n} \sum_{j=1}^n \max\{x_j^{p_i}, (\log n)^{2p_i}\}, \ \ i=1,\ldots,k.
\]
Noting that $g_i$ is bounded by $(\log n)^{2p_i}$, the proof can be completed as before. 
\end{proof}

\section{Proof of Theorem \ref{thm2}}
In this section we deal with the case $b >2$. As usual, $C$ will denote any constant that depends only on $b$. We also set $q := \sqrt{b-2}$, a constant that will occur often. 

The proof of the localization draws inspiration from Talagrand's localization theorem for the $p$-spin Hopfield model (see  Section 5.11 of \cite{talagrand03}). 

\begin{proof}[Proof of parts $(c)$ and $(d)$] Let $Y_1,Y_2,\ldots$ be i.i.d.\ $Exp(1)$ random variables, and let $Y=(Y_1,\ldots,Y_n)$. Take $\ep = n^{-10}$ as before.
Let $a := 1/100$. Let $Z_i = Y_i1_{\{Y_i \le n^a\}}$ and let $v := \ee(Z_1^2)$. By Hoeffding's inequality, we have
\[
\pp\biggl(\biggl|\sum_{i=1}^n (Z_i^2  -v)\biggr| > t\biggr) \le 2e^{-t^2/2n^{1+4a}}. 
\]
Thus, if we define
\[
A := \biggl\{\biggl|\sum_{i=1}^n (Z_i^2  -v)\biggr| > n^{5/6}\biggr\},
\]
then
\[
\pp(A)\le 2\exp\biggl(-\frac{n^{\frac{2}{3}-4a}}{2}\biggr).
\]
Next, let $B$ be the event that there is a set $I\subseteq \{1,\ldots,n\}$ of size $k := [n^{(1-a)/2}]$ such that $Y_i > n^a$ for all $i\in I$. Then
\begin{align*}
\pp(B) \le {n \choose k} e^{-kn^a}\le n^k e^{-kn^a}\le C\exp\biggl(-\frac{n^{(1+a)/2}}{C}\biggr).
\end{align*}
Let $D$ be the event that $\sum_{i\in I} Y_i > qn^{1/2} + n^{(2-a)/4}$ for some subset $I\subseteq \{1,\ldots,n\}$ of size $< k$. Note that for any $j$, $\sum_{i=1}^j Y_i$ follows a $Gamma(j,1)$ distribution. Therefore, for any $j\ge 2$, $t > 2$,
\begin{align*}
\pp\biggl(\sum_{i=1}^jY_i > t\biggr) &= \int_t^\infty \frac{x^{j-1}}{(j-1)!} e^{-x} dx \\
&= e^{-t}\int_0^\infty \frac{(x+t)^{j-1}}{(j-1)!} e^{-x} dx \\
&\le e^{-t}\int_0^\infty \frac{2^{j-2}(x^{j-1}+t^{j-1})}{(j-1)!} e^{-x} dx\\
&\le e^{-t}(2^{j-2} + t^{j-1}) \le 2t^{j-1} e^{-t}. 
\end{align*}
(Note that the inequality is also true for $j=1$.) Thus, if  $n$ is sufficiently large so that $k < n/2$ and we take $t := qn^{1/2} + n^{(2-a)/4}$, then  
\begin{align*}
\pp(D) &\le \sum_{j=1}^{k-1}{n\choose j} 2t^{j-1} e^{-t} \le Ckn^kt^k e^{-t}. 
\end{align*}
Since $k = [n^{(1-a)/2}]$ and $(1-a)/2 < (2-a)/4$, we see that
\[
\pp(D) \le C\exp\biggl(-qn^{1/2} - \frac{n^{(2-a)/4}}{C}\biggr).
\]
Now suppose $A^c \cap B^c \cap D^c \cap \{Y\in K^\ep\}$ happens. Let $I$ be the set of $i$ such that $Y_i > n^a$. Since $B^c$ has happened, therefore $|I| < k$. Since $D^c$ has occurred, we must have that
\begin{equation}\label{leftbd}
\sum_{i\in I} Y_i \le qn^{1/2} + n^{(2-a)/4}. 
\end{equation}
Again, since $Y\in K^\ep$, we have 
\[
\biggl|\sum_{i=1}^n Y_i^2 - bn\biggr| < nb\ep = bn^{-9}. 
\]
But due to $A^c$, 
\[
\biggl|\sum_{i\not\in I} Y_i^2 - vn\biggr|\le n^{5/6}. 
\]
Combining the last two inequalities, we get
\begin{align*}
\biggl|\sum_{i\in I} Y_i^2 - (b-v)n\biggr|&\le b n^{-9} + n^{5/6}\le Cn^{5/6}. 
\end{align*}
But
\[
v = \int_0^{n^a} x^2 e^{-x} dx = 2- \int_{n^a}^\infty x^2 e^{-x} dx, 
\]
and therefore
\[
|v-2|\le Ce^{-n^a/C}. 
\]
Thus, under $A^c \cap B^c\cap D^c \cap \{Y\in K^\ep\}$, 
\[
\biggl|\sum_{i\in I} Y_i^2 - (b-2)n\biggr|\le  Cn^{5/6}. 
\]
Let $M^Y := \max_i Y_i$. The above inequality combined with \eqref{leftbd} shows that under  $A^c \cap B^c\cap C^c \cap \{Y\in K^\ep\}$, we have
\begin{align*}
q^2n - Cn^{5/6} &\le \sum_{i\in I} Y_i^2 \\
&\le M^Y \sum_{i\in I} Y_i\le M^Y(qn^{1/2} + n^{(2-a)/4}).
\end{align*}
Therefore, since $a/4 < 1/6$, 
\begin{align*}
M^Y &\ge \frac{q^2n - Cn^{5/6}}{qn^{1/2} + n^{(2-a)/4}} \\
&= qn^{1/2}\frac{1-Cn^{-1/6}}{1+n^{-a/4}}\ge qn^{1/2} (1- C n^{- a/4}). 
\end{align*}
But under $D^c$ we have
\[
M^Y \le qn^{1/2} + n^{(2-a)/4}.
\]
Thus, under $A^c \cap B^c\cap D^c \cap \{Y\in K^\ep\}$, we have
\[
|M^Y - qn^{1/2}|\le C n^{(2-a)/4}. 
\]
Therefore, from the bounds on $\pp(A), \pp(B),\pp(D)$ obtained above (and observing that the bound on $\pp(D)$ dominates the other two), we get 
\begin{equation}\label{mbd}
\begin{split}
&\pp(|M^Y - qn^{1/2}| > C n^{(2-a)/4}, \; Y\in K^\ep) \\
&\le \pp(A\cup B\cup D) \\
&\le C\exp\biggl(-qn^{1/2} - \frac{n^{(2-a)/4}}{C}\biggr). 
\end{split}
\end{equation}
Let $M^Y_2$ be the second largest among the $Y_i$'s. Then either $M^Y_2 < n^a$, or under $A^c\cap B^c\cap D^c \cap \{Y\in K^\ep\}$, 
\begin{align*}
M^Y_2 &\le \sum_{i\in I} Y_i - M^Y\\
&\le qn^{1/2} + n^{(2-a)/4} - (qn^{1/2} - C n^{(2-a)/4})\\
&= C n^{(2-a)/4}. 
\end{align*}
Thus, again, we have
\begin{equation}\label{m2bd}
\begin{split}
&\pp(M^Y_2 > Cn^{(2-a)/4}, \; Y\in K^\ep) \\
&\le C\exp\biggl(-qn^{1/2} - \frac{n^{(2-a)/4}}{C}\biggr). 
\end{split}
\end{equation}
This gives us the bounds on the numerator in Proposition \ref{propbd2}, except that we have to evaluate the Lipschitz constant $L$ for $M$ and $M_2$. For a vector~$x$, let $g_1(x)$ and $g_2(x)$ denote the largest and second-largest components of~$x$. Since 
\[
|g_1(x)-g_1(y)| = |\max_i x_i - \max_i y_i|\le \max_i|x_i - y_i|,
\]
it follows that we can take $L=1$ for $g_1$. By the same logic,
\begin{align*}
|\max_{i<j}(x_i+x_j) - \max_{i<j}(y_i+y_j)|&\le \max_{i<j}|(x_i + x_j) - (y_i + y_j)| \\
&\le 2\max_i|x_i - y_i|.
\end{align*} 
However, $\max_{i<j}(x_i + x_j) = g_1(x)+g_2(x)$. Thus, we can take $L=3$ for $g_2$. Thus by \eqref{mbd}, \eqref{m2bd} and Proposition \ref{propbd2},  we have
\begin{equation}\label{mbdmain}
\begin{split}
&\pp(|M- qn^{1/2}|> C n^{(2-a)/4}, \; Y\in K^\ep) \\
&\le \frac{C \exp\bigl( - qn^{1/2} - C^{-1}n^{(2-a)/4}\bigr)}{\pp(m(Y) > c(b)\ep,\; Y\in K^\ep)}
\end{split}
\end{equation}
and 
\begin{equation}\label{m2bdmain}
\pp(M_2> C n^{(2-a)/4}, \; Y\in K^\ep) \le \frac{C \exp\bigl( - qn^{1/2} - C^{-1}n^{(2-a)/4}\bigr)}{\pp(m(Y) > c(b)\ep,\; Y\in K^\ep)}.
\end{equation}
Let us now start working on the denominator in the above expressions. 
Let $\delta := C(b) \ep$. Note that
\begin{equation}\label{trick}
\begin{split}
\pp(m(Y) > \delta, \; Y\in K^\ep)&= \pp(Y\in K^\ep \mid m(Y) > \delta) \pp(m(Y) >\delta)\\
&= \pp(Y\in K^\ep \mid m(Y) > \delta) e^{-\delta n}. 
\end{split}
\end{equation}
Since $Y_1,\ldots,Y_n$ are i.i.d.\ $Exp(1)$, it follows from the memoryless property of the exponential distribution that the conditional distribution of $Y$ given $m(Y) > \delta$ is the same as the unconditional distribution of $Y + \delta {\bf 1}$. Thus,
\begin{equation}\label{trick2}
\pp(Y\in K^\ep \mid m(Y) > \delta) = \pp(Y+\delta {\bf 1} \in K^\ep). 
\end{equation}
Note that 
\[
\mu(Y+\delta{\bf 1}) = \mu(Y) + \delta, \ \ \mu_2(Y+\delta{\bf 1}) = \mu_2(Y) + 2\delta\mu(Y) + \delta^2. 
\]
Let 
\[
E := {\textstyle\biggl\{\bigl|\mu(Y) - \bigl(1-\delta+\frac{3}{2}\ep\bigr)\bigr|< \ep^2, \ \bigl|\mu_2(Y) - \bigl(b - 2\delta + \frac{b+1}{2}\ep\bigr) \bigr|<\ep^2\biggr\}}. 
\]
If $E$ happens, then 
\[
{\textstyle 1+\frac{3}{2}\ep - \ep^2} < \mu(Y) + \delta < {\textstyle 1+\frac{3}{2}\ep + \ep^2},
\]
and thus, if $n$ is sufficiently large (so that $\ep^2 = n^{-20}\ll \ep$), we have
\begin{equation}\label{e1}
1+\ep < \mu(Y+ \delta{\bf 1}) < 1+2\ep. 
\end{equation}
Again, under $E$, we have
\begin{align*}
&{\textstyle\bigl|\mu_2(Y) + 2\delta\mu(Y)+\delta^2 - (b+\frac{b+1}{2}\ep)\bigr| }\\
&\le {\textstyle\bigl|\mu_2(Y) - (b-2\delta+\frac{b+1}{2}\ep)\bigr|} + 2\delta|\mu(Y)-1| + \delta^2 \\
&\le C\ep^2.
\end{align*}
Thus, if $n$ is sufficiently large, and $E$ happens, then we have
\begin{equation}\label{e2}
\begin{split}
b + \ep &< {\textstyle b +\frac{b+1}{2}\ep- C\ep^2} \\
&< \mu_2(Y + \delta{\bf 1}) < {\textstyle b +\frac{b+1}{2}\ep+ C\ep^2}< b+ b\ep. 
\end{split}
\end{equation}
By \eqref{e1} and \eqref{e2}, we see that $E$ implies $Y+\delta{\bf 1} \in K^\ep$, provided $n$ is large enough. Now let
\[
\mu^-(Y) := \frac{1}{n}\sum_{i=2}^n Y_i, \ \ \mu_2^-(Y) := \frac{1}{n}\sum_{i=2}^n Y_i^2. 
\]
Define
\begin{align*}
E' &:= {\textstyle\biggl\{\bigl|\mu^-(Y) - \bigl(1-\delta+\frac{3}{2}\ep - qn^{-1/2}\bigr)\bigr|< \frac{1}{2}\ep^2\biggr\}}\\
&\qquad \cap{\textstyle \biggl\{\bigl|\mu_2^-(Y) - \bigl(2 - 2\delta + \frac{b+1}{2}\ep\bigr) \bigr|<\frac{1}{2}\ep^2\biggr\}}\\
&\qquad \cap {\textstyle\bigl\{\bigl|Y_1^2 - q^2n\bigr| < \frac{1}{2}\ep^2\bigr\}}. 
\end{align*}
Suppose $E'$ happens. Then
\begin{align*}
&{\textstyle \bigl|\mu_2(Y) - \bigl(b - 2\delta + \frac{b+1}{2}\ep\bigr) \bigr| }\\
&\le {\textstyle \bigl|\mu_2^-(Y) - \bigl(2 - 2\delta + \frac{b+1}{2}\ep\bigr) \bigr| + \frac{1}{n}|Y_1^2 - (b-2)n|}\\
&< {\textstyle \frac{1}{2}\ep^2 + \frac{1}{2n}\ep^2 \le \ep^2.} 
\end{align*}
Again, under $E'$, 
\begin{align*}
\bigl|Y_1 - qn^{1/2}\bigr| &= \frac{|Y_1^2 - q^2 n|}{Y_1 + qn^{1/2}}\le Cn^{-1/2}\ep^2,
\end{align*}
and therefore, for sufficiently large $n$,
\begin{align*}
& {\textstyle\bigl|\mu(Y) - \bigl(1-\delta+\frac{3}{2}\ep\bigr)\bigr|}\\
&\le  {\textstyle\bigl|\mu^-(Y) - \bigl(1-\delta+\frac{3}{2}\ep - qn^{-1/2}\bigr)\bigr| + \frac{1}{n}\bigl|Y_1 - qn^{1/2}\bigr|}\\
&< {\textstyle \frac{1}{2}\ep^2 + Cn^{-3/2}\ep^2 \le \ep^2}.
\end{align*}
Thus, $E'$ implies $E$. Since $Y_1$ is independent of $(Y_2,\ldots, Y_n)$ and $\ee(Y_i)=1$, $\ee(Y_i^2)=2$, we can apply Lemma \ref{local} to the pair $(\mu^-(Y), \mu_2^-(Y))$ conclude that
\begin{align*}
\pp(E') &\ge C^{-1} n\ep^4 {\textstyle \pp(|Y_1^2 - q^2n|< \frac{1}{2}\ep^2)}\\
&\ge C^{-1} n\ep^4 {\textstyle \pp(|Y_1 - qn^{1/2}|< C^{-1}n^{-1/2}\;\ep^2)}\\
&\ge C^{-1}n^{1/2}\ep^6 e^{-qn^{1/2}}. 
\end{align*}
(The second inequality holds because $|Y_1^2-q^2n| \le (Y_1-qn^{1/2})^2 + 2qn^{1/2}|Y_1-qn^{1/2}|$.) Therefore by \eqref{trick} and \eqref{trick2},
\begin{equation}\label{low22}
\begin{split}
\pp(m(Y) > \delta, \; Y\in K^\ep) &= \pp(Y\in K^\ep \mid m(Y) > \delta)e^{-\delta n}\\
&= \pp(Y+\delta{\bf 1}\in K^\ep) e^{-\delta n}\\
&\ge \pp(E)e^{-\delta n}\ge \pp(E')e^{-\delta n} \ge C^{-1} n^{-60} e^{-qn^{1/2}}.
\end{split}
\end{equation}
Combining this with \eqref{mbdmain} and Proposition \ref{propbd2} (and the value of $L$ obtained before), we get
\[
\pp(|M- qn^{1/2}| > Cn^{(2-a)/4})\le Ce^{-C^{-1}n^{(2-a)/4}}. 
\] 
Similarly from \eqref{m2bdmain} we get 
\[
\pp(M_2 > Cn^{(2-a)/4}) \le C e^{-C^{-1}n^{(2-a)/4}}.
\]
This completes the proof of parts $(c)$ and $(d)$.
\end{proof}
\begin{proof}[Proof of parts $(a)$ and $(b)$]
Part $(b)$ is obvious by symmetry. So we only have to prove part $(a)$. We proceed exactly as in the proof of part $(a)$ in Theorem \ref{thm1}. Let $h:\rr \ra \rr$ be a function satisfying $|h(x)|\le 1$ and $|h(x)-h(y)|\le L|x-y|$ for all $x,y\in \rr$, where $L$ is some positive constant. Define $g:\rr^n \ra \rr$ as 
\[
g(x) := \frac{1}{n}\sum_{i=1}^n h(x_i). 
\]
Setting $a :=\ee h(Y_1)$ and using Hoeffding's inequality, we get
\[
\pp(|g(Y)-a|> t) \le 2e^{-nt^2/2}.
\]
However, the lower bound \eqref{low22} for $\pp(m(Y)>C(b)\ep, \; Y\in K^\ep)$ is different from \eqref{ylow}. Using \eqref{low22} and Proposition \eqref{propbd2}, we get the following analog of~\eqref{tails}:
\begin{align*}
\pp(|g(X)-a| > t) &\le C e^{C \sqrt{n}}e^{-n(t-C_3(b)L\ep n)^2/2}.
\end{align*}
The tail bound decays rapidly in the regime $t > Cn^{-1/4}$. This gives
\[
\ee|g(X)-a| \le Cn^{-1/4}.
\]
As before, by Jensen's inequality we get
\[
|\ee h(X_1) -a|\le Cn^{-1/4}. 
\]
The joint distribution of $(X_1,\ldots, X_k)$ is handled similarly. 
\end{proof}

\section{Proof of Theorem \ref{thm3}}
The proof of Theorem \ref{thm3} is basically contained in the earlier proofs. Fix $x_0 > 0$ and $1< b\le 2$. In the proof of part $(a)$ of Theorem \ref{thm1}, if instead of taking a fixed $h$ let us take
\[
h_n(x) := 
\begin{cases}
1 & \text{ if } x < x_0,\\
1 - (x-x_0)n^{5} &\text{ if } x_0 \le x < x_0+n^{-5},\\
0 &\text{ if } x \ge x_0 + n^{-5},
\end{cases}
\]
then $|h_n(x)|\le 1$ and $|h_n(x)-h_n(y)|\le n^5 |x-y|$ for all $x,y$. Hereafter we can proceed exactly as in the proof of \eqref{tails} (taking $L = n^5$) and conclude that
\begin{equation}\label{hnbd}
|\ee h_n(X_1) - \ee h_n(Z_1)|\le C\sqrt{\frac{\log n}{n}}. 
\end{equation}
Since $Z_1$ has a bounded density, this gives 
\begin{align*}
\pp(X_1 \le x_0) &\le \ee h_n(X_1) \\
&\le \ee h_n(Z_1) + C\sqrt{\frac{\log n}{n}}\\
&\le \pp(Z_1 \le x_0) + Cn^{-5} + C\sqrt{\frac{\log n}{n}}. 
\end{align*}
Next, let us slightly modify the definition of $h_n$ by  replacing $x_0$ with $x_0 - n^{-5}$. Let us call the new function $\tilde{h}_n$. Then  \eqref{hnbd} holds for $\tilde{h}_n$ too, and hence
\begin{align*}
\pp(X_1 \le x_0) &\ge \ee \tilde{h}_n(X_1)\\
&\ge \ee\tilde{h}_n(Z_1) - C\sqrt{\frac{\log n}{n}}\\
&\ge \pp(Z_1 \le x_0) - Cn^{-5} - C\sqrt{\frac{\log n}{n}}.
\end{align*}
This completes the proof for $k=1$. The general case is similar, as in the proof of Theorem \ref{thm1}. When $b >2$, the proof is exactly the same, except that the bound in \eqref{hnbd} becomes $Cn^{-1/4}$, as in the proof of part $(a)$ of Theorem~\ref{thm2}.  

\vskip.2in
\noindent{\bf Acknowledgment.} The author thanks Persi Diaconis for bringing this problem to his attention and Julien Barr\'e for posing the problem to Persi, Johel Beltran and Wilson Cabanillas for pointing our a small error, and Assaf Naor and the anonymous referee for helpful comments.

\end{document}